\documentclass[paper=a4, headings=small]{scrartcl}

\usepackage[T1]{fontenc}
\usepackage[utf8]{inputenc}
\usepackage[english]{babel}
\usepackage{graphicx}
\usepackage{amsmath}
\usepackage{amssymb}
\usepackage{amstext}
\usepackage[amsmath,thmmarks,thref]{ntheorem}
\usepackage{tikz}
\usetikzlibrary{shapes}
\usepackage{xcolor}
\usepackage{hyperref}

\theoremseparator{.}
\newtheorem{thm}{Theorem}
\newtheorem{lem}[thm]{Lemma}

\newtheorem{prop}[thm]{Proposition}

\theorembodyfont{\upshape}

\theoremstyle{nonumberplain}
\theoremheaderfont{\itshape}
\theoremseparator{.}
\theoremsymbol{\ensuremath{_\square}}
\newtheorem{proof}{Proof}

\newcommand{\mathU}{\mathcal{U} }
\newcommand{\U}{$\mathcal{U}$ }

\title{Is the right-angled building associated to a universal group unique?}
\author{Lara Be{\ss}mann \thanks{Funded as project KR1668/10 by the Deutsche Forschungsgemeinschaft, and under Germany's Excellence Strategy EXC 2044-390685587, Mathematics M\"unster: Dynamics-Geometry-Structure. This work is part of the author's PhD thesis. }}

\begin{document}
\maketitle 
\begin{abstract}
A universal group is a subgroup of the group of type preserving automorphisms of a right-angled building and hence associated to this building. 
A question is then if this universal group can act chamber-transitively and with compact open stabilisers on a different right-angled building of the same type.
We answer this question and define two universal groups associated to different right-angled buildings which are isomorphic as topological groups.  
\end{abstract}

\section{Introduction}
A universal group associated to a right-angled building is a subgroup of the group of type-preserving automorphisms of the building, such that the local actions are given by prescribed permutation groups. 
These groups have been introduced for trees by Burger and Mozes in \cite{BM1} and generalised to right-angled buildings by De Medts, Silva, and Struyve in  \cite{TomAnaKoen}.

When studying  universal groups associated to right-angled buildings, the question arises if two universal groups associated to different right-angled buildings can be isomorphic.

We will answer this question and describe two universal groups associated to different right-angled buildings that are isomorphic as topological groups. 
By choosing the thicknesses the right way, we obtain two different buildings of the same type but in such a way that we also obtain a bijection of chambers. 
To get this bijection we use the so-called tree-wall tree. 
Then we use this bijection of chambers to obtain a bijection of the associated universal groups.
For this we choose the local groups in the right way. 
This bijection is then continuous and open and hence the universal groups are isomorphic as topological groups. 

We start with defining universal groups.
A building of type $(W,I)$ is right-angled if the Coxeter system $(W,I)$ is right-angled, i.e. $m_{ij} \in \{2, \infty\}$ for all $i,j \in I$ with $i \neq j$.
  
Let $\Delta$ be a right-angled building of type $(W,I)$ with prescribed thickness $(q_i)_{i \in I}$ such that the $q_i$ are finite but at least 3 for every $i \in I$. 
Then for every $i \in I$ every $i$-panel contains exactly $q_i$ many chambers.
The set of chambers of $\Delta$ is denoted by $Ch(\Delta)$.  
We view automorphisms of $\Delta$ as maps from $Ch(\Delta)$ to $Ch(\Delta)$ such that $i$-adjacency is preserved for all $i \in I$. 
For every $i \in I$ let $X_i$ be a set of $i$-colours with cardinality $q_i$. 
Let
\begin{align*}
\lambda \colon Ch(\Delta) \to \prod_{i \in I} X_i, c \mapsto (\lambda_i(c))_{i \in I}
\end{align*}
be a map, such that $\lambda_i \vert_{\mathcal{P}} \colon \mathcal{P} \to X_i$ is bijective and $\lambda_j\vert_{\mathcal{P}} \colon \mathcal{P} \to X_j$ is constant for every $i \in I$, every $j \in I$ with $j \neq i$, and every $i$-panel $\mathcal{P}$. 
We call $\lambda$ then a \textit{legal colouring} of $\Delta$. 

The map 
\begin{align*}
\sigma_{\lambda}(g,\mathcal{P}) = \lambda_i\vert_{g(\mathcal{P})} \circ g\vert_{\mathcal{P}} \circ (\lambda_i\vert_{\mathcal{P}})^{-1}
\end{align*}
is called the \textit{local action} of the automorphism $g \in Aut(\Delta)$ on the $i$-panel $\mathcal{P}$. 

The universal group is the subgroup of the automorphism group such that all local actions are contained in prescribed permutation groups. 
Let $F_i \subseteq Sym(X_i)$ be a permutation group, for every $i \in I$.
We refer to these groups as the \textit{local groups}. 
Then the \textit{universal group of $\Delta$ with respect to the groups $(F_i)_{i \in I}$} is defined as
\begin{align*}
\mathU = \left\{ g \in Aut(\Delta) \mid \sigma_{\lambda}(g, \mathcal{P}) \in F_i \text{ for every } i \in I \text{ and every $i$-panel } \mathcal{P} \right\}.
\end{align*}

These universal groups have been introduced and studied in \cite{TomAnaKoen} and have been further studied in \cite{TomJens}. 
The properties of the universal group depend on and correspond to the properties of the local groups. 
By Proposition 3.7 in \cite{TomAnaKoen} the universal groups for different legal colourings are conjugate in $Aut(\Delta)$ and hence the structure of \U does not depend on the choice of the colouring. 

We equip the automorphism group of a building with the topology of pointwise convergence. 
The universal group carries then the subspace topology and a neighbourhood basis of the identity is given by pointwise stabilisers of finite sets of chambers in the universal group. 
If $q_i$ is finite for every $i \in I$, then every pointwise stabiliser of a finite set of chambers has finite orbits and thus is compact.  
Hence, the universal group is a locally compact totally disconnected group (cf. \cite[Prop. 2.2 and 3.8]{TomJens}). 

\section{Two isomorphic universal groups}
We want to answer the question if a universal group can act chamber-transitively and with compact open stabilisers on a different building of the same type as the original one. 
For this we construct two universal groups associated to different buildings of the same type and show that they are isomorphic as topological groups. 

So let $\Delta$ and $\tilde{\Delta}$ be two locally finite thick buildings of type $(W,I)$ where $I=\{i,j,k\}$ and $W=\left\langle i,j,k \mid i^2 = j^2 = k^2 =(ij)^2=1 \right\rangle$.
Moreover, we prescribe the thickness $(q_i,q_k,q_j)$ for $\Delta$ and $(\tilde{q}_i,\tilde{q}_k,\tilde{q}_j)$ for $\tilde{\Delta}$
such that the equations
\begin{align*}
q_i q_j = \tilde{q}_k \quad\text{and}\quad \tilde{q}_i \tilde{q}_j = q_k
\end{align*}
are satisfied.
Parts of the Davis realisations of two buildings of this type, which satisfy the assumptions are pictured in Figure \ref{DavRealExample}. 
\begin{figure} \centering
First, a part of the Davis realisation of $\Delta$ with thickness $q_i=4, q_j=3$ and $q_k=9$. 

\begin{tikzpicture}[kpoint/.style={circle,draw,fill=lightgray,inner sep=1.5pt},
ipoint/.style={circle,draw,fill=darkgray,inner sep=1.5pt},
jpoint/.style={circle,draw,fill=gray,inner sep=1.5pt},
chamberpoint/.style={circle=4,draw,fill=black,inner sep=1pt},
					ijpoint/.style={diamond,draw,fill=black,inner sep=1.5pt},scale=1.8]

\coordinate (pkc1) at (0,0);

\coordinate (c9) at (1,0);
\coordinate (c2) at (0.77,0.64);
\coordinate (c3) at (0.17,0.98);
\coordinate (c4) at (-0.5,0.87);
\coordinate (c5) at (-0.93,0.34);
\coordinate (c6) at (-0.93,-0.34);
\coordinate (c7) at (-0.5,-0.87);
\coordinate (c8) at (0.17,-0.98);
\coordinate (c1) at (0.77,-0.64);

\coordinate (pic1) at (1.77,-0.64);
\coordinate (pjc1) at (0.77,-1.64);
\coordinate (rijc1) at (1.77,-1.64);
\coordinate (c2pic1) at (2.77,-0.64);
\coordinate (pjc2pic1) at (2.77,-1.64);
\coordinate (c3pic1) at (1.77+0.87,-0.64-0.5);
\coordinate (pjc3pic1) at (1.77+0.87,-0.64-0.5-1);
\coordinate (c4pic1) at (1.77+0.59,-0.64+0.81);
\coordinate (pjc4pic1) at (1.77+0.59,-0.64+0.81-1);
\coordinate (c2pjc1) at (0.77+0.38, -1.64-0.92);
\coordinate (pic2pjc1) at (0.77+0.38+1, -1.64-0.92);
\coordinate (c3pjc1) at (0.77,-1.64-1);
\coordinate (pic3pjc1) at (0.77+1,-1.64-1);

\coordinate (c2pjc2pic1) at (2.77+0.38, -1.64-0.92);
\coordinate (c3pjc2pic1) at (2.77, -1.64-1);

\coordinate (c2pjc3pic1) at (1.77+0.87+0.38,-0.64-0.5-1-0.92);
\coordinate (c3pjc3pic1) at (1.77+0.87, -1.64-1-0.5);

\coordinate (c2pjc4pic1) at (1.77+0.59+0.38,-0.64+0.81-1-0.92);
\coordinate (c3pjc4pic1) at (1.77+0.59, -1.64-1+0.81);

\coordinate (pkc2pic1) at (2.77+0.92,-0.64+0.38);
\coordinate (pkc3pic1) at (1.77+0.87+0.92,-0.64-0.5+0.38);
\coordinate (pkc4pic1) at (1.77+0.59+0.92,-0.64+0.81+0.38);
\coordinate (pkc2pjc1) at (0.77+0.38-0.92, -1.64-0.92+0.38);
\coordinate (pkc3pjc1) at (0.77-0.92,-1.64-1-0.38);
\coordinate (pkc2pjc2pic1) at (2.77+0.5+0.92, -1.64-0.87-0.38);
\coordinate (pkc3pjc2pic1) at (2.77+0.92, -1.64-1-0.38);
\coordinate (pkc2pjc3pic1) at (1.77+0.87+0.38+0.92,-0.64-0.5-1-0.92-0.38);
\coordinate (pkc3pjc3pic1) at (1.77+0.87+0.92, -1.64-1-0.5-0.38);
\coordinate (pkc2pjc4pic1) at (1.77+0.59+0.38+0.92,-0.64+0.81-1-0.92-0.38);
\coordinate (pkc3pjc4pic1) at (1.77+0.59+0.92, -1.64-1+0.81-0.38);

\draw (pkc1) -- (c1);
\draw (pkc1) -- (c2);
\draw (pkc1) -- (c3);
\draw (pkc1) -- (c4);
\draw (pkc1) -- (c5);
\draw (pkc1) -- (c6);
\draw (pkc1) -- (c7);
\draw (pkc1) -- (c8);
\draw (pkc1) -- (c9);

\draw (c1) -- (pic1);
\draw (c2pic1) -- (pic1);
\draw (c3pic1) -- (pic1);
\draw (c4pic1) -- (pic1);
\draw (c1) -- (pjc1);
\draw (c2pjc1) -- (pjc1);
\draw (c3pjc1) -- (pjc1); 
\draw (c2pjc1) -- (pic2pjc1);
\draw (c3pjc1) -- (pic3pjc1);
\draw (c2pic1) -- (pjc2pic1);
\draw (c3pic1) -- (pjc3pic1);
\draw (c4pic1) -- (pjc4pic1);
\draw (pic1) -- (rijc1);
\draw (pjc1) -- (rijc1);
\draw (pic2pjc1) -- (rijc1);
\draw (pic3pjc1) -- (rijc1);
\draw (pjc2pic1) -- (rijc1);
\draw (pjc3pic1) -- (rijc1);
\draw (pjc4pic1) -- (rijc1);

\draw (pjc2pic1) -- (c2pjc2pic1);
\draw (pjc2pic1) -- (c3pjc2pic1);
\draw (pic3pjc1) -- (c3pjc2pic1);
\draw (pic2pjc1) -- (c2pjc2pic1);

\draw (pjc3pic1) -- (c2pjc3pic1);
\draw (pic2pjc1) -- (c2pjc3pic1);
\draw (pjc3pic1) -- (c3pjc3pic1);
\draw (pic3pjc1) -- (c3pjc3pic1);

\draw (pjc4pic1) -- (c2pjc4pic1);
\draw (pic2pjc1) -- (c2pjc4pic1);
\draw (pjc4pic1) -- (c3pjc4pic1);
\draw (pic3pjc1) -- (c3pjc4pic1);

\draw (pkc2pic1) -- (c2pic1);
\draw (pkc3pic1) -- (c3pic1); 
\draw (pkc4pic1) -- (c4pic1);
\draw (pkc2pjc1) -- (c2pjc1);
\draw (pkc3pjc1) -- (c3pjc1);
\draw (pkc2pjc2pic1) -- (c2pjc2pic1); 
\draw (pkc3pjc2pic1) -- (c3pjc2pic1);
\draw (pkc2pjc3pic1) -- (c2pjc3pic1);
\draw (pkc3pjc3pic1) -- (c3pjc3pic1); 
\draw (pkc2pjc4pic1) -- (c2pjc4pic1);
\draw (pkc3pjc4pic1) -- (c3pjc4pic1); 

\draw node[chamberpoint] at (c1){};
\draw node[chamberpoint] at (c2){};
\draw node[chamberpoint] at (c3){};
\draw node[chamberpoint] at (c4){};
\draw node[chamberpoint] at (c5){};
\draw node[chamberpoint] at (c6){};
\draw node[chamberpoint] at (c7){};
\draw node[chamberpoint] at (c8){};
\draw node[chamberpoint] at (c9){};
\draw node[chamberpoint] at (c2pic1){};
\draw node[chamberpoint] at (c3pic1){};
\draw node[chamberpoint] at (c4pic1){};
\draw node[chamberpoint] at (c2pjc1){};
\draw node[chamberpoint] at (c3pjc1){};
\draw node[chamberpoint] at (c2pjc2pic1){};
\draw node[chamberpoint] at (c2pjc3pic1){};
\draw node[chamberpoint] at (c3pjc2pic1){};
\draw node[chamberpoint] at (c3pjc3pic1){};
\draw node[chamberpoint] at (c2pjc4pic1){};
\draw node[chamberpoint] at (c3pjc4pic1){};

\draw node[kpoint] at (pkc1){};
\draw node[kpoint] at (pkc2pic1){};
\draw node[kpoint] at (pkc3pic1){};
\draw node[kpoint] at (pkc4pic1){};
\draw node[kpoint] at (pkc2pjc1){};
\draw node[kpoint] at (pkc3pjc1){};
\draw node[kpoint] at (pkc2pjc2pic1){};
\draw node[kpoint] at (pkc2pjc3pic1){};
\draw node[kpoint] at (pkc3pjc2pic1){};
\draw node[kpoint] at (pkc3pjc3pic1){};
\draw node[kpoint] at (pkc2pjc4pic1){};
\draw node[kpoint] at (pkc3pjc4pic1){};

\draw node[ipoint] at (pic1){};
\draw node[ipoint] at (pic2pjc1){};
\draw node[ipoint] at (pic3pjc1){};

\draw node[jpoint] at (pjc1){};
\draw node[jpoint] at (pjc2pic1){};
\draw node[jpoint] at (pjc3pic1){};
\draw node[jpoint] at (pjc4pic1){};

\draw node[ijpoint] at (rijc1){};

\draw node at (2.77+0.92+0.2,-0.64+0.38){$\dots$};
\draw node at (1.77+0.87+0.92+0.2,-0.64-0.5+0.38){$\dots$};
\draw node at (1.77+0.59+0.92+0.2,-0.64+0.81+0.38){$\dots$};
\draw node at (0.77+0.38-0.92-0.2, -1.64-0.92+0.38){$\dots$};
\draw node at (0.77-0.92-0.2,-1.64-1-0.38){$\dots$};
\draw node at (2.77+0.5+0.92+0.2, -1.64-0.87-0.38){$\dots$};
\draw node at (2.77+0.92+0.2, -1.64-1-0.38){$\dots$};
\draw node at (1.77+0.87+0.38+0.92+0.2,-0.64-0.5-1-0.92-0.38){$\dots$};
\draw node at (1.77+0.87+0.92+0.2, -1.64-1-0.5-0.38){$\dots$};
\draw node at (1.77+0.59+0.38+0.92+0.2,-0.64+0.81-1-0.92-0.38){$\dots$};
\draw node at (1.77+0.59+0.92+0.2, -1.64-1+0.81-0.38){$\dots$};

\draw node at (-0.5-0.2,0.87+0.1){$\dots$};
\draw node at (-0.5-0.1,-0.87-0.2){$\dots$};
\draw node at (0.87+0.1,0.5+0.2){$\dots$};
\end{tikzpicture}

Second, a part of the Davis realisation of $\tilde{\Delta}$ with thickness $\tilde{q}_i=3, \tilde{q}_j=3$ and $\tilde{q}_k=12$.

\begin{tikzpicture}[kpoint/.style={circle,draw,fill=lightgray,inner sep=1.5pt},
ipoint/.style={circle,draw,fill=darkgray,inner sep=1.5pt},
jpoint/.style={circle,draw,fill=gray,inner sep=1.5pt},
chamberpoint/.style={circle,draw,fill=black,inner sep=1pt},
					ijpoint/.style={diamond,draw,fill=black,inner sep=1.5pt},scale=1.8]

\coordinate (pkc1) at (0,0);

\coordinate (c12) at (1,0);
\coordinate (c2) at (0.87,0.5);
\coordinate (c3) at (0.5,0.87);
\coordinate (c4) at (0,1);
\coordinate (c5) at (-0.5,0.87);
\coordinate (c6) at (-0.87,0.5);
\coordinate (c7) at (-1,0);
\coordinate (c8) at (-0.87,-0.5);
\coordinate (c9) at (-0.5,-0.87);
\coordinate (c10) at (0,-1);
\coordinate (c11) at (0.5,-0.87);
\coordinate (c1) at (0.87,-0.5);

\coordinate (pic1) at (1.87,-0.5);
\coordinate (pjc1) at (0.87,-1.5);
\coordinate (rijc1) at (1.87,-1.5);
\coordinate (c2pic1) at (2.87,-0.5);
\coordinate (pjc2pic1) at (2.87,-1.5);
\coordinate (c3pic1) at (1.87+0.7,-0.5+0.7);
\coordinate (pjc3pic1) at (1.87+0.7,-0.5+0.7-1);
\coordinate (c2pjc1) at (0.87+0.5, -1.5-0.87);
\coordinate (pic2pjc1) at (0.87+0.5+1, -1.5-0.87);
\coordinate (c3pjc1) at (0.87,-1.5-1);
\coordinate (pic3pjc1) at (0.87+1,-1.5-1);

\coordinate (c2pjc2pic1) at (2.87+0.5, -1.5-0.87);
\coordinate (c3pjc2pic1) at (2.87, -1.5-1);

\coordinate (c2pjc3pic1) at (1.87+0.7+0.5,-0.5+0.7-1-0.87);
\coordinate (c3pjc3pic1) at (1.87+0.7, -1.5-1+0.7);

\coordinate (pkc2pic1) at (2.87+0.92,-0.5+0.38);
\coordinate (pkc3pic1) at (1.87+0.7+0.92,-0.5+0.7+0.38);
\coordinate (pkc2pjc1) at (0.87+0.5-0.92, -1.5-0.87+0.38);
\coordinate (pkc3pjc1) at (0.87-0.92,-1.5-1-0.38);
\coordinate (pkc2pjc2pic1) at (2.87+0.5+0.92, -1.5-0.87-0.38);
\coordinate (pkc3pjc2pic1) at (2.87+0.92, -1.5-1-0.38);
\coordinate (pkc2pjc3pic1) at (1.87+0.7+0.5+0.92,-0.5+0.7-1-0.87-0.38);
\coordinate (pkc3pjc3pic1) at (1.87+0.7+0.92, -1.5-1+0.7-0.38);

\draw (pkc1) -- (c1);
\draw (pkc1) -- (c2);
\draw (pkc1) -- (c3);
\draw (pkc1) -- (c4);
\draw (pkc1) -- (c5);
\draw (pkc1) -- (c6);
\draw (pkc1) -- (c7);
\draw (pkc1) -- (c8);
\draw (pkc1) -- (c9);
\draw (pkc1) -- (c10);
\draw (pkc1) -- (c11);
\draw (pkc1) -- (c12);

\draw (c1) -- (pic1);
\draw (c2pic1) -- (pic1);
\draw (c3pic1) -- (pic1);
\draw (c1) -- (pjc1);
\draw (c2pjc1) -- (pjc1);
\draw (c3pjc1) -- (pjc1); 
\draw (c2pjc1) -- (pic2pjc1);
\draw (c3pjc1) -- (pic3pjc1);
\draw (c2pic1) -- (pjc2pic1);
\draw (c3pic1) -- (pjc3pic1);
\draw (pic1) -- (rijc1);
\draw (pjc1) -- (rijc1);
\draw (pic2pjc1) -- (rijc1);
\draw (pic3pjc1) -- (rijc1);
\draw (pjc2pic1) -- (rijc1);
\draw (pjc3pic1) -- (rijc1);

\draw (pjc2pic1) -- (c2pjc2pic1);
\draw (pjc2pic1) -- (c3pjc2pic1);
\draw (pic3pjc1) -- (c3pjc2pic1);
\draw (pic2pjc1) -- (c2pjc2pic1);

\draw (pjc3pic1) -- (c2pjc3pic1);
\draw (pic2pjc1) -- (c2pjc3pic1);
\draw (pjc3pic1) -- (c3pjc3pic1);
\draw (pic3pjc1) -- (c3pjc3pic1);

\draw (pkc2pic1) -- (c2pic1);
\draw (pkc3pic1) -- (c3pic1); 
\draw (pkc2pjc1) -- (c2pjc1);
\draw (pkc3pjc1) -- (c3pjc1);
\draw (pkc2pjc2pic1) -- (c2pjc2pic1); 
\draw (pkc3pjc2pic1) -- (c3pjc2pic1);
\draw (pkc2pjc3pic1) -- (c2pjc3pic1);
\draw (pkc3pjc3pic1) -- (c3pjc3pic1); 

\draw node[chamberpoint] at (c1){};
\draw node[chamberpoint] at (c2){};
\draw node[chamberpoint] at (c3){};
\draw node[chamberpoint] at (c4){};
\draw node[chamberpoint] at (c5){};
\draw node[chamberpoint] at (c6){};
\draw node[chamberpoint] at (c7){};
\draw node[chamberpoint] at (c8){};
\draw node[chamberpoint] at (c9){};
\draw node[chamberpoint] at (c10){};
\draw node[chamberpoint] at (c11){};
\draw node[chamberpoint] at (c12){};
\draw node[chamberpoint] at (c2pic1){};
\draw node[chamberpoint] at (c3pic1){};
\draw node[chamberpoint] at (c2pjc1){};
\draw node[chamberpoint] at (c3pjc1){};
\draw node[chamberpoint] at (c2pjc2pic1){};
\draw node[chamberpoint] at (c2pjc3pic1){};
\draw node[chamberpoint] at (c3pjc2pic1){};
\draw node[chamberpoint] at (c3pjc3pic1){};

\draw node[kpoint] at (pkc1){};
\draw node[kpoint] at (pkc2pic1){};
\draw node[kpoint] at (pkc3pic1){};
\draw node[kpoint] at (pkc2pjc1){};
\draw node[kpoint] at (pkc3pjc1){};
\draw node[kpoint] at (pkc2pjc2pic1){};
\draw node[kpoint] at (pkc2pjc3pic1){};
\draw node[kpoint] at (pkc3pjc2pic1){};
\draw node[kpoint] at (pkc3pjc3pic1){};

\draw node[ipoint] at (pic1){};
\draw node[ipoint] at (pic2pjc1){};
\draw node[ipoint] at (pic3pjc1){};

\draw node[jpoint] at (pjc1){};
\draw node[jpoint] at (pjc2pic1){};
\draw node[jpoint] at (pjc3pic1){};

\draw node[ijpoint] at (rijc1){};

\draw node at (2.87+0.92+0.2,-0.5+0.38){$\dots$};
\draw node at (1.87+0.7+0.92+0.2,-0.5+0.7+0.38){$\dots$};
\draw node at (0.87+0.5-0.92-0.2, -1.5-0.87+0.38){$\dots$};
\draw node at (0.87-0.92-0.2,-1.5-1-0.38){$\dots$};
\draw node at (2.87+0.5+0.92+0.2, -1.5-0.87-0.38){$\dots$};
\draw node at (2.87+0.92+0.2, -1.5-1-0.38){$\dots$};
\draw node at (1.87+0.7+0.5+0.92+0.2,-0.5+0.7-1-0.87-0.38){$\dots$};
\draw node at (1.87+0.7+0.92+0.2, -1.5-1+0.7-0.38){$\dots$};
\draw node at (-0.5-0.2,0.87+0.1){$\dots$};
\draw node at (-0.5-0.1,-0.87-0.2){$\dots$};
\draw node at (0.87+0.1,0.5+0.2){$\dots$};

\end{tikzpicture}

In both pictures the {\begin{tikzpicture}[kpoint/.style={circle,draw,fill=lightgray,inner sep=1.5pt}] \draw node[kpoint] at (0,0){}; \end{tikzpicture}}-vertices correspond to $k$-panels, the {\begin{tikzpicture}[ipoint/.style={circle,draw,fill=darkgray,inner sep=1.5pt}]  \draw node[ipoint] at (0,0){}; \end{tikzpicture}}-vertices to $i$-panels, the {\begin{tikzpicture}[jpoint/.style={circle,draw,fill=gray,inner sep=1.5pt}] \draw node[jpoint] at (0,0){}; \end{tikzpicture}}-vertices to $j$-panels and the {\begin{tikzpicture}[ijpoint/.style={diamond,draw,fill=black,inner sep=1.5pt}] \draw node[ijpoint] at (0,0){}; \end{tikzpicture}}-vertices to residues of type $\{i,j\}$. 
The {\begin{tikzpicture}[chamberpoint/.style={circle,draw,fill=black,inner sep=1pt}] \draw node[chamberpoint] at (0,0){}; \end{tikzpicture}}-vertices correspond to chambers. 
All squares are filled. 
\caption{An example of a pair of buildings satisfying the assumptions.}\label{DavRealExample}
\end{figure}

At first, we want to construct a bijection of chambers from $Ch(\Delta)$ to $Ch(\tilde{\Delta})$.
For this we use tree-wall trees, which have been studied for universal groups in section 2.3 in \cite{TomAnaKoen}.  
The $k$-tree-wall tree associated to a building of this type is the following infinite graph. 
The vertices are the $k$-panels, which are residues of type $k$, and the residues of type $\{i,j\}$, and there is an edge whenever the intersection of the residues is not empty. 
So there are no edges between different $k$-panels and between different residues of type $\{i,j\}$. 
The intersection of a $k$-panel and a residue of type $\{i,j\}$  is either empty or a single chamber. 
Moreover, every chamber is contained in exactly one $k$-panel and one residue of type $\{i,j\}$.
Hence, the edges are in bijection with the chambers of the building. 
The graph is a biregular tree with valencies the cardinality of the $k$-panels and of the residues of type $\{i,j\}$ and by Proposition 2.39 in \cite{TomAnaKoen} it is indeed a tree. 
An example of a $k$-tree-wall tree is pictured in Figure \ref{ExampleTreeWallTree}.
\begin{figure}[ht]\centering
\begin{tikzpicture}[kpoint/.style={circle,draw,fill=lightgray,inner sep=1.5pt},
					ijpoint/.style={diamond,draw,fill=black,inner sep=1.5pt},scale=2]
\coordinate (ij1) at (0,0);

\coordinate (k1) at (1,0);
\coordinate (k2) at (0.87,0.5);
\coordinate (k3) at (0.5,0.87);
\coordinate (k4) at (0,1);
\coordinate (k5) at (-0.5,0.87);
\coordinate (k6) at (-0.87,0.5);
\coordinate (k7) at (-1,0);
\coordinate (k8) at (-0.87,-0.5);
\coordinate (k9) at (-0.5,-0.87);
\coordinate (k10) at (0,-1);
\coordinate (k11) at (0.5,-0.87);
\coordinate (k12) at (0.87,-0.5);

\coordinate (k2ij1) at (0.87+1,0.5+0);
\coordinate (k2ij2) at (0.87+0.98,0.5+0.17);
\coordinate (k2ij3) at (0.87+0.94,0.5+0.34);
\coordinate (k2ij4) at (0.87+0.87,0.5+0.5);
\coordinate (k2ij5) at (0.87+0.77,0.5+0.64);
\coordinate (k2ij6) at (0.87+0.64,0.5+0.77);
\coordinate (k2ij7) at (0.87+0.5,0.5+0.87);
\coordinate (k2ij8) at (0.87+0.34,0.5+0.94);

\coordinate (k6ij1) at (-0.87-0.5,0.5+0.87);
\coordinate (k6ij2) at (-0.87-0.64,0.5+0.77);
\coordinate (k6ij3) at (-0.87-0.77,0.5+0.64);
\coordinate (k6ij4) at (-0.87-0.87,0.5+0.5);
\coordinate (k6ij5) at (-0.87-0.94,0.5+0.34);
\coordinate (k6ij6) at (-0.87-0.99,0.5+0.17);
\coordinate (k6ij7) at (-0.87-1,0.5+0);
\coordinate (k6ij8) at (-0.87-0.98,0.5-0.17);

\coordinate (k2ij3k1) at (0.87+0.94+0.99,0.5+0.34-0.14);
\coordinate (k2ij3k2) at (0.87+0.94+1,0.5+0.34+0);
\coordinate (k2ij3k3) at (0.87+0.94+0.99,0.5+0.34+0.14);
\coordinate (k2ij3k4) at (0.87+0.94+0.96,0.5+0.34+0.28);
\coordinate (k2ij3k5) at (0.87+0.94+0.91,0.5+0.34+0.41);
\coordinate (k2ij3k6) at (0.87+0.94+0.85,0.5+0.34+0.53);
\coordinate (k2ij3k7) at (0.87+0.94+0.77,0.5+0.34+0.64);
\coordinate (k2ij3k8) at (0.87+0.94+0.67,0.5+0.34+0.74);
\coordinate (k2ij3k9) at (0.87+0.94+0.56,0.5+0.34+0.83);
\coordinate (k2ij3k10) at (0.87+0.94+0.43,0.5+0.34+0.9);
\coordinate (k2ij3k11) at (0.87+0.94+0.96,0.5+0.34-0.28);

\draw (ij1) -- (k1);
\draw (ij1) -- (k2);
\draw (ij1) -- (k3);
\draw (ij1) -- (k4);
\draw (ij1) -- (k5);
\draw (ij1) -- (k6);
\draw (ij1) -- (k7);
\draw (ij1) -- (k8);
\draw (ij1) -- (k9);
\draw (ij1) -- (k10);
\draw (ij1) -- (k11);
\draw (ij1) -- (k12);

\draw (k2) -- (k2ij1);
\draw (k2) -- (k2ij2);
\draw (k2) -- (k2ij3);
\draw (k2) -- (k2ij4);
\draw (k2) -- (k2ij5);
\draw (k2) -- (k2ij6);
\draw (k2) -- (k2ij7);
\draw (k2) -- (k2ij8);

\draw (k6) -- (k6ij1);
\draw (k6) -- (k6ij2);
\draw (k6) -- (k6ij3);
\draw (k6) -- (k6ij4);
\draw (k6) -- (k6ij5);
\draw (k6) -- (k6ij6);
\draw (k6) -- (k6ij7);
\draw (k6) -- (k6ij8);

\draw (k2ij3) -- (k2ij3k1);
\draw (k2ij3) -- (k2ij3k2);
\draw (k2ij3) -- (k2ij3k3);
\draw (k2ij3) -- (k2ij3k4);
\draw (k2ij3) -- (k2ij3k5);
\draw (k2ij3) -- (k2ij3k6);
\draw (k2ij3) -- (k2ij3k7);
\draw (k2ij3) -- (k2ij3k8);
\draw (k2ij3) -- (k2ij3k9);
\draw (k2ij3) -- (k2ij3k10);
\draw (k2ij3) -- (k2ij3k11);

\draw node[ijpoint] at (ij1){};

\draw node[kpoint] at (k1){};
\draw node[kpoint] at (k2){};
\draw node[kpoint] at (k3){};
\draw node[kpoint] at (k4){};
\draw node[kpoint] at (k5){};
\draw node[kpoint] at (k6){};
\draw node[kpoint] at (k7){};
\draw node[kpoint] at (k8){};
\draw node[kpoint] at (k9){};
\draw node[kpoint] at (k10){};
\draw node[kpoint] at (k11){};
\draw node[kpoint] at (k12){};

\draw node[ijpoint] at (k2ij1){};
\draw node[ijpoint] at (k2ij2){};
\draw node[ijpoint] at (k2ij3){};
\draw node[ijpoint] at (k2ij4){};
\draw node[ijpoint] at (k2ij5){};
\draw node[ijpoint] at (k2ij6){};
\draw node[ijpoint] at (k2ij7){};
\draw node[ijpoint] at (k2ij8){};

\draw node[ijpoint] at (k6ij1){};
\draw node[ijpoint] at (k6ij2){};
\draw node[ijpoint] at (k6ij3){};
\draw node[ijpoint] at (k6ij4){};
\draw node[ijpoint] at (k6ij5){};
\draw node[ijpoint] at (k6ij6){};
\draw node[ijpoint] at (k6ij7){};
\draw node[ijpoint] at (k6ij8){};

\draw node[kpoint] at (k2ij3k1){};
\draw node[kpoint] at (k2ij3k2){};
\draw node[kpoint] at (k2ij3k3){};
\draw node[kpoint] at (k2ij3k4){};
\draw node[kpoint] at (k2ij3k5){};
\draw node[kpoint] at (k2ij3k6){};
\draw node[kpoint] at (k2ij3k7){};
\draw node[kpoint] at (k2ij3k8){};
\draw node[kpoint] at (k2ij3k9){};
\draw node[kpoint] at (k2ij3k10){};
\draw node[kpoint] at (k2ij3k11){};

\draw node[] at (1.2,-0.7){$\dots$};
\draw node[] at (-1.2,-0.75){$\dots$};
\draw node[] at (-2.18,0.68){$\dots$};
\draw node[] at (-0.2,1.2){$\dots$};
\draw node[] at (2.92,1.39){$\dots$};
\draw node[] at (1.57,1.58){$\dots$};
\end{tikzpicture}

A part of the $k$-tree-wall tree associated to $\Delta$ with thickness $q_i=4, q_j=3$ and $q_k=9$. The 
\begin{tikzpicture}
\draw node[circle,draw,fill=lightgray,inner sep=1.5pt] at (0,0){};
\end{tikzpicture}
-vertices correspond to $k$-panels and the 
\begin{tikzpicture}
\draw node[diamond,draw,fill=black,inner sep=1.5pt] at (0,0){};
\end{tikzpicture}
-vertices to residues of type $\{i,j\}$.
\caption{Example of a $k$-tree-wall tree.}\label{ExampleTreeWallTree} 
\end{figure}

The $k$-tree-wall trees associated to $\Delta$ and $\tilde{\Delta}$ are both $(q_k,\tilde{q}_k)$-regular trees, hence they are isomorphic and so we get a bijection of chambers
\begin{align*}
\psi \colon Ch(\Delta) \to Ch(\tilde{\Delta}),
\end{align*}
that maps $k$-panels of $\Delta$ to residues of type $\{i,j\}$ in $\tilde{\Delta}$ and the residues of type $\{i,j\}$ in $\Delta$ to the $k$-panels of $\tilde{\Delta}$.
We get the equivalences
\begin{align*}
c \sim_k d &\Leftrightarrow \psi(c) \in Res_{\{i,j\}}(\psi(d)) \\
c \in Res_{\{i,j\}}(d) &\Leftrightarrow \psi(c) \sim_k \psi(d). 
\end{align*}
for any chambers $c,d \in Ch(\Delta)$.

Now we need colourings for both buildings. 
Let $X_i$ and $X_j$ be sets of cardinality $q_i$ and $q_j$ respectively, and $\tilde{X}_i$ and $\tilde{X}_j$ be sets of cardinality $\tilde{q}_i$ and $\tilde{q}_j$ and define $X_k = \tilde{X}_i \times \tilde{X}_j$ and $\tilde{X}_k = X_i \times X_j$. 
Then $X_k$ has cardinality $q_k = \tilde{q}_i \tilde{q}_j$ and $\tilde{X}_k$ has cardinality $\tilde{q}_k = q_i q_j$.
Let $\lambda_i \colon Ch(\Delta) \to X_i$ be a map, such that the restriction of the map to every $i$-panel in $\Delta$ is a bijection and to every panel of different type is constant. 
Define $\lambda_j \colon Ch(\Delta) \to X_j, \tilde{\lambda}_i \colon Ch(\tilde{\Delta}) \to \tilde{X}_i$ and $\tilde{\lambda}_j \colon Ch(\tilde{\Delta}) \to \tilde{X}_j$ in the same way. 
We use these maps and the bijection of chambers to define the $k$-colourings in the following way
\begin{align*}
\lambda_k \colon Ch(\Delta) &\to X_k , c \mapsto \left( \tilde{\lambda}_i(\psi(c)), \tilde{\lambda}_j(\psi(c)) \right), \\
\tilde{\lambda}_k \colon Ch(\tilde{\Delta}) &\to \tilde{X}_k, \tilde{c} \mapsto \left( \lambda_i(\psi^{-1}(\tilde{c})), \lambda_j(\psi^{-1}(\tilde{c})) \right).
\end{align*}
Since $\tilde{\lambda}_i$ and $\tilde{\lambda}_j$ are bijective, and the image of a $k$-panel is a residue of type $\{i,j\}$, which is finite, the restriction of $\lambda_k$ to a $k$-panel is a bijection.
The image of a panel of type $i$ or $j$ is contained in a $k$-panel and since $\tilde{\lambda}_i$ and $\tilde{\lambda}_j$ are constant on $k$-panels, the restriction of $\lambda_k$ to panels of type $i$ and $j$ is constant.
Analogously, it follows that the restriction of $\tilde{\lambda}_k$ to a $k$-panel is bijective and to a panel of type $i$ or $j$ constant.  
Hence, we get legal colourings $\lambda \colon Ch(\Delta) \to X_i \times X_j \times X_k$ of $\Delta$ and $\tilde{\lambda} \colon Ch(\tilde{\Delta}) \to \tilde{X}_i \times \tilde{X}_j \times \tilde{X}_k$ of $\tilde{\Delta}$.

Let $F_i \subseteq Sym(X_i), F_j \subseteq Sym(X_j), \tilde{F}_i \subseteq Sym(\tilde{X}_i)$ and $\tilde{F}_j \subseteq Sym(\tilde{X}_j)$ be transitive subgroups. 
It is not necessary to assume transitivity, however we want to consider chamber-transitive universal groups and therefore we need to consider transitive local groups. 
Define $F_k = \tilde{F}_i \times \tilde{F}_j$ and $\tilde{F}_k = F_i \times F_j$, then $F_k$ acts transitively on $X_k$ and $\tilde{F}_k$ transitively on $\tilde{X}_k$. 
Let \U be the universal group of $\Delta$ with respect to $(F_i, F_j, F_k)$ and let $\tilde{\mathU}$ be the universal group of $\tilde{\Delta}$ with respect to $(\tilde{F}_i, \tilde{F}_j, \tilde{F}_k)$. 
We will prove that these universal groups \U and $\tilde{\mathU}$ are isomorphic as topological groups.
We start with setting the automorphisms of $\Delta$ in relation to those of $\tilde{\Delta}$ by using the bijection of chambers. 

\begin{lem}
For every $g \in \mathU$ the conjugate $\psi g \psi^{-1}$ is an automorphism of $\tilde{\Delta}$ and contained in the universal group $\tilde{\mathU}$. 
\end{lem}
\begin{proof}
First, we need to prove that $\psi g \psi^{-1}$ is an automorphism of $\tilde{\Delta}$. 
Let $\tilde{c}, \tilde{d} \in Ch(\tilde{\Delta})$. 
Assume that $\psi g \psi^{-1}(\tilde{c}) = \psi g \psi^{-1}(\tilde{d})$, since $\psi$ and $g$ are bijections of chambers it follows that $\tilde{c}=\tilde{d}$. 
Furthermore, $\psi g^{-1} \psi^{-1}(\tilde{c})$ is a chamber of $\tilde{\Delta}$ and a preimage of $\tilde{c}$. 
Hence, $\psi g \psi^{-1}$ is a bijection of chambers.

Further, we need to prove that $\psi g \psi^{-1}$ acts type-preservingly. 
For this, let $\tilde{c} \sim_i \tilde{d}$. 
Then by the definition of the colouring, they have the same $j$-colour. 
Moreover, by the definition of $\psi$, their preimages $\psi^{-1}(\tilde{c})$ and $\psi^{-1}(\tilde{d})$ lie in the same $k$-panel of $\Delta$. 
This adjacency is preserved by $g$ and again by the definition of $\psi$ it follows that $\psi g \psi^{-1}(\tilde{c})$ and $\psi g \psi^{-1}(\tilde{d})$ are in the same residue of type $\{i,j\}$ in $\tilde{\Delta}$. 

But since $g \in \mathU$ and $\psi^{-1}(\tilde{c})$ and $\psi^{-1}(\tilde{d})$ are $k$-adjacent, there are $\sigma_1 \in \tilde{F}_i$ and $\sigma_2 \in \tilde{F}_j$ such that $\sigma_{\lambda}(g, \mathcal{P}_{k,\psi^{-1}(\tilde{c})})=(\sigma_1,\sigma_2)$. 
Hence, it follows that 
\begin{align*}
(\sigma_1, \sigma_2)\left(\lambda_k(\psi^{-1}(\tilde{c}))\right) 
= \lambda_k(g \psi^{-1}(\tilde{c})) \text{ and }
(\sigma_1, \sigma_2)\left(\lambda_k(\psi^{-1}(\tilde{d}))\right) 
= \lambda_k(g \psi^{-1}(\tilde{d}))
\end{align*}
and moreover, by the definition of the $k$-colouring of $\Delta$, it follows that
\begin{align*}
\left(\sigma_1(\tilde{\lambda}_i(\tilde{c})), \sigma_2(\tilde{\lambda}_j(\tilde{c})) \right)
&= \left( \tilde{\lambda}_i(\psi g \psi^{-1}(\tilde{c})), \tilde{\lambda}_j(\psi g \psi^{-1}(\tilde{c})) \right) \text{ and } \\
\left(\sigma_1(\tilde{\lambda}_i(\tilde{d})), \sigma_2(\tilde{\lambda}_j(\tilde{d})) \right)
&=\left( \tilde{\lambda}_i(\psi g \psi^{-1}(\tilde{d})), \tilde{\lambda}_j(\psi g \psi^{-1}(\tilde{d})) \right).
\end{align*}   
By using that $\tilde{c}$ and $\tilde{d}$ have the same $j$-colour, we conclude that 
\begin{align*}
\tilde{\lambda}_j(\psi g \psi^{-1} (\tilde{c})) = \sigma_2 \tilde{\lambda}_j(\tilde{c}) = \sigma_2 \tilde{\lambda}_j(\tilde{d})= \tilde{\lambda}_j(\psi g \psi^{-1} (\tilde{d})) 
\end{align*}
and hence $\psi g \psi^{-1} (\tilde{c})$ and $\psi g \psi^{-1} (\tilde{d})$ have the same $j$-colour and are contained in the same residue of type $\{i,j\}$. 
It follows that they are contained in the same $i$-panel. 

If $\tilde{c}$ and $\tilde{d}$ are contained in the same $j$-panel it follows analogously that their images $\psi g \psi^{-1}(\tilde{c})$ and $\psi g \psi^{-1}(\tilde{d})$ are also contained in the same $j$-panel. 

Let now $\tilde{c} \sim_k \tilde{d}$. 
Then by definition of $\psi$ the preimages $\psi^{-1}(\tilde{c})$ and $\psi^{-1}(\tilde{d})$ lie in the same residue of type $\{i,j\}$ in $\Delta$.
Since adjacency is preserved by $g$ and again by the definition of $\psi$, we get that $\psi g \psi^{-1}(\tilde{c}) \sim _k \psi g \psi^{-1}(\tilde{d})$. 
So we conclude that $\psi g \psi^{-1}$ is an automorphism of $\tilde{\Delta}$. 

Second, we need to prove that $\psi g \psi^{-1}$ is contained in the universal group $\tilde{\mathU}$. 
For this we consider the local actions and show that they are contained in the local groups. 
Let $\tilde{\mathcal{P}}$ be an $i$-panel of $\tilde{\Delta}$, then we get
\begin{align*}
\sigma_{\tilde{\lambda}}\left( \psi g\psi^{-1}, \tilde{\mathcal{P}} \right) 
&= \tilde{\lambda}_i \circ \psi g\psi^{-1} \circ (\tilde{\lambda}_i)^{-1} \vert_{\tilde{\mathcal{P}}} \\
&= pr_1 \lambda_k g (\lambda_k)^{-1} \vert _{\psi^{-1}(\tilde{\mathcal{P}})} \\
&= pr_1 \sigma_{\lambda}\left( g , Res_k({\psi^{-1}(\tilde{\mathcal{P}})}) \right) \in \tilde{F}_i. 
\end{align*}
If $\tilde{\mathcal{P}}$ is a $j$-panel we get again $\sigma_{\tilde{\lambda}}( \psi g\psi^{-1}, \tilde{\mathcal{P}}) \in \tilde{F}_j$.  
Assume now that $\tilde{\mathcal{P}}$ is a $k$-panel. 
Then we get
\begin{align*}
\sigma_{\tilde{\lambda}}\left( \psi g\psi^{-1}, \tilde{\mathcal{P}} \right) 
&= \tilde{\lambda}_k \circ \psi g\psi^{-1} \circ (\tilde{\lambda}_k)^{-1} \vert_{\tilde{\mathcal{P}}} \\
&=  (\lambda_i g (\lambda_i)^{-1}, \lambda_j g (\lambda_j)^{-1})\vert_{\psi^{-1}(\tilde{\mathcal{P}})}  \in F_i \times F_j = \tilde{F}_k.
\end{align*}
Hence we conclude that $\psi g\psi^{-1} \in \tilde{\mathU}$. 
\end{proof}
Completely analogously, we get that $\psi^{-1} \tilde{g} \psi \in \mathU$ for every automorphism $\tilde{g} \in \tilde{\mathU}$.
So we can define the following map 
\begin{align*}
\varphi \colon \mathU \to \tilde{\mathU}, g \mapsto \psi g \psi^{-1}.
\end{align*}
It follows immediately that $\varphi$ is a group homomorphism. 
\begin{prop}
The group homomorphism $\varphi \colon \mathU \to \tilde{\mathU}, g \mapsto \psi g \psi^{-1}$ is a bijection.
\end{prop}
\begin{proof}
For every $\tilde{g} \in \tilde{\mathU}$ the automorphism $\psi^{-1} \tilde{g} \psi$ is contained in \U and moreover $\varphi(\psi^{-1} \tilde{g} \psi)= \tilde{g}$. 
Hence, $\varphi$ is surjective. 

Let $g,h \in \mathU$ with $\varphi(g)= \varphi(h)$.
Then $\psi g \psi^{-1} = \psi h \psi^{-1}$ and thus $\psi g \psi^{-1}(\tilde{c}) = \psi h \psi^{-1}(\tilde{c})$ for every chamber $\tilde{c} \in Ch(\tilde{\Delta})$. 
Since $\psi$ is a bijection of chambers it follows that $g(c)=h(c)$ for every chamber $c \in Ch(\Delta)$ and hence $\varphi$ is injective. 
\end{proof}
So we conclude that \U and $\tilde{\mathU}$ are isomorphic. 

Moreover, we want to prove that the topologies of \U and $\tilde{\mathU}$ also coincide and hence we need to prove first that $\varphi$ is continuous and then that it is open, to conclude that the topologies coincide. 
\begin{prop}
The isomorphism $\varphi$ is continuous.
\end{prop}
\begin{proof}
Let $W \subseteq \tilde{\mathU}$ be a neighbourhood of the identity. 
Hence, there exists a finite set of chambers $D \subseteq Ch(\tilde{\Delta})$ such that the pointwise stabiliser of $D$ in $\tilde{\mathU}$ is contained in $W$. % $Fix_{\tilde{\mathU}}(D) \subseteq W$. 
Let $g \in Fix_{{\mathU}}(\psi^{-1}(D))$, then it follows that $\psi g \psi^{-1}(d) = d$ for every $d \in D$ and hence $\varphi(g) \in Fix_{\tilde{\mathU}}(D)$. 

It follows that the pointwise stabiliser of $\psi^{-1}(D)$ is contained in the preimage of the pointwise stabiliser of $D$ and since $\psi^{-1}(D)$ is finite,  $Fix_{\mathU}(\psi^{-1}(D))$ is an identity neighbourhood and contained in $\varphi^{-1}(W)$. 
Hence, $\varphi$ is continuous at the identity and thus everywhere. 
\end{proof}

\begin{prop}
The continuous isomorphism $\varphi$ is open and hence a homeomorphism. 
\end{prop}
\begin{proof}
Let $V \subseteq \mathU$ be an open set. 
Let $v \in V$, then exists a finite set of chambers $C \subseteq Ch(\Delta)$ such that $vFix_{\mathU}(C) \subseteq V$. 
Since $\varphi$ is bijective, there exists for every $h \in Fix_{\tilde{\mathU}}(\psi(C))$ a preimage $g \in \mathU$ with $\varphi(g)=h$. 
For every $c \in C$ we get $\psi(c) = h\psi(c)=\psi g \psi^{-1} \psi(c)=\psi g(c)$ and hence $g \in Fix_{\mathU}(C)$. 
Then the pointwise stabiliser of $\psi(C)$ % $Fix_{\tilde{\mathU}}(\psi(C))$ 
is contained in the image of the pointwise stabiliser of $C$ under $\varphi$ % $\varphi(Fix_{\mathU}(C))$ 
and moreover $\varphi(v)Fix_{\tilde{\mathU}}(\psi(C)) \subseteq \varphi(V)$. 
Hence, the image $\varphi(V)$ is open. 
\end{proof}

So all in all, we proved that \U acts via $\varphi$ chamber-transitively and with compact open stabilisers on the building $\tilde{\Delta}$. 
From this, we can conclude that the building that is associated to a universal group may not be the unique right-angled building on which \U acts chamber-transitively and with compact open stabilisers.

Note that the topological properties of the universal groups depend on the properties of the local groups. 
If at least one local group does not act freely, then the universal group is not discrete. 
Hence, we obtain in this way topologically isomorphic non-discrete locally compact universal groups.  

This can be generalised to every thick semi-regular locally finite right-angled building of type $(W,I)$ if there exists $k \in I$ such that $I-\{k\}$ is spherical and $kj \neq jk$ for every $j \in I-\{k\}$.  

\subsection*{Acknowledgments}
I would like to thank my advisor Linus Kramer for many helpful remarks and suggestions, and Philip Möller and Daniel Keppeler for useful comments on an earlier version of this manuscript.

\end{document}